\documentclass[a4, 12pt]{amsart}

\usepackage{amssymb,amsthm,amstext,amsmath,amscd,latexsym,amsfonts,amscd}
\usepackage{color}

\def\m{\mathfrak m}
\def\p{\mathfrak p}

\def\w{\omega}

\def\D{{\mathcal D}}
\def\L{{\mathcal L}}
\def\Hom{\mathrm{Hom}}
\def\RHom{\mathrm{RHom}}
\def\pHom{\underline{\mathrm{Hom}}}
\def\Lotimes{{\, {}^{\mathrm L}\otimes}}
\def\Ker{\mathrm{Ker}}

\def\Cok{\mathrm{Coker}}
\def\Im{\mathrm{Im}}
\def\tr{\mathrm{Tr}}

\def\depth{\mathrm{depth}}
\def\Ext{\mathrm{Ext}}
\def\Tor{\mathrm{Tor}}

\def\CM{\mathrm{CM}}
\def\pCM{\underline{\mathrm{CM}}}

\def\NF{\mathrm{NF}}

\def\RGamma{{\mathrm{R}\Gamma}}
\def\Spec{\mathrm{Spec}}
\def\Supp{\mathrm{Supp}}
\def\supp{\mathrm{supp}}

\def\mono{\hookrightarrow}

\theoremstyle{plain} 
\newtheorem{theorem}{\bfseries Theorem}[section]
\newtheorem{lemma}[theorem]{\textbf Lemma}
\newtheorem{corollary}[theorem]{\textbf Corollary}
\newtheorem{proposition}[theorem]{\textbf Proposition}

\theoremstyle{definition} 
\newtheorem{definition}[theorem]{\textbf Definition}
\newtheorem{remark}[theorem]{\textbf Remark}

\newtheorem{question}[theorem]{\textbf Question}
\newtheorem{conjecture}[theorem]{\textbf Conjecture}

\theoremstyle{plain}

\begin{document}

\title{An Auslander-Reiten principle in derived categories}
\author{Maiko Ono}
\author{Yuji Yoshino}

\address{Department of Mathematics, Okayama University, 700-8530, Okayama, Japan.}
\email{ptr56u9x@s.okayama-u.ac.jp}

\address{Department of Mathematics, Okayama University, 700-8530, Okayama, Japan.}
\email{yoshino@math.okayama-u.ac.jp}
\date{\today}

\pagestyle{plain}


\maketitle
\thispagestyle{empty} 


\section{Introduction}
This paper concerns commutative Noetherian rings. 

Let  $R$  be a Cohen-Macaulay local ring of Krull dimension $d$  with canonical module $\w$  and let $M, N$  be maximal Cohen-Macaulay modules over $R$. 
Assume that $R$ has only an isolated singularity. 
Then we have an isomorphism
\begin{equation}\label{ar}
\pHom _R (M, N) ^{\vee} \cong \Ext^1_R (N, \tau M), 
\end{equation}
where $\tau M = \Hom _R( \Omega ^d(\tr M), \w)$. 
For the definition of $\tr M$, see the paragraph preceding Corollary \ref{ARcor}.
This isomorphism (\ref{ar}) is known as Auslander-Reiten duality, or simply AR duality. 
For the proof of (\ref{ar}) the reader should refer to \cite[Proposition 1.1]{Aus3}. 

The AR duality plays a crucial role in the theory of maximal Cohen-Macaulay modules.
In fact, one can derive from (\ref{ar}) the existence of Auslander-Reiten sequence in the category of maximal Cohen-Macaulay modules over an isolated singularity. 
See \cite[Theorem 3.2]{Ybook}. 
Further assuming that $R$ is Gorenstein, it assures us that the stable category of the category of maximal Cohen-Macaulay modules has $(d-1)$-Calabi-Yau property.
See \cite[Theorem 8.3]{IY}. 
 
Recently, Iyama and Wemyss \cite{IW} have generalized the AR duality to rings whose singular locus has at most one dimension. 
See \cite[Theorem 3.1]{IW}. 

The purpose of this paper is to propose a general principle behind the AR duality, by which we mean a general theorem for modules or chain complexes of modules in a kind of general form that encompasses the classical AR duality and its generalization. 
In the end we have reached the following conclusion to this aim of building the principle, which we dare call the AR principle :  

\vspace{6pt} 
\noindent 
{\bf Theorem[AR Principle]}(Theorem \ref{AR}).  
{\it 
Let $R$ be a commutative Noetherian ring and let $W$ be a specialization-closed subset of $\Spec (R)$. 
Given a bounded complex $I$ of injective $R$-modules with $I^i=0$ for all $i>n$ and a complex $X$ such that the support of $H^i(X)$ is contained in $W$ for all $i<0$, 
the natural map $\Gamma_W I\rightarrow I $ induces isomorphisms 

$$
\Ext^i _R(X, \Gamma_W I ) \overset{\cong}{\longrightarrow} \Ext^i _R (X, I)\,\,\,\,\,\,\,\,\,\text{for } i>n. 
$$
}

This result is proved in \S 2.
We emphasize that this theorem is similar to a version of the local duality theorem; see Remark \ref{LD}.

In \S3 we apply the AR principle to deduce the formula  (\ref{ar}).  
See Corollary \ref{ARcor}. 
In fact, we consider the case where $(R, \m)$  is a local ring,  $W = \{ \m \}$, and  $I$  is a dualizing complex of  $R$.
Then it naturally induces Theorem \ref{AR1} below, which is also regarded as a generalization of the original AR duality (\ref{ar}). 

In almost the same circumstances above but  $W = \{ \p\in\Spec(R)\,|\,\dim  R/\p \leqq 1 \}$, we deduce from AR principle the generalization of AR duality due to Iyama and Wemyss. 
This will be explained in detail in \S 4. 
See Theorem \ref{AR3} and Corollary \ref{AR4} in particular. 

In \S 5 we discuss the Auslander-Reiten conjecture for modules over Gorenstein rings.
The Auslander-Reiten conjecture (abbreviated to ARC) can be stated in its most general form as follows:

\vspace{6pt}
\noindent
{\bf(ARC)}\; 
Let $R$ be a commutative Noetherian ring and $M$ a finitely generated $R$-module. 
If $\Ext _{R}^{i}({M}, {M{\oplus}R})=0$ for all $i>0$, then $M$ is projective.
\vspace{6pt}

For the history of ARC and some recent results on ARC, see \cite{A},\cite{Aus1},\cite{Aus2},\cite{HSV} and \cite{SE}.
By virtue of our AR principle we can prove a more stronger result than (ARC) in some cases. 
Actually Corollary \ref{end} below gives the following : 

\vspace{6pt}
\noindent 
{\bf Theorem} (See Corollary \ref{end}).
Let $R$ be a Gorenstein local ring of dimension $d $ that is larger than $2$. 
Assume that $M$  is a maximal Cohen-Macaulay $R$-module whose non-free locus has dimension $\leqq 1$, i.e. 
$M _{\p}$  is  $R_{\p}$-free for any $\p \in \Spec (R)$ with $\dim{R/\p} > 1$. 
Furthermore we assume that
$$
\Ext _R^{d-1}(M, M) = 0=\Ext ^{d-2}_{R}(M, M).
$$
Then M is a free $R$-module.
\vspace{12pt}
\par
\noindent
{\bf Acknowledgement.}
The authors are very grateful to the anonymous referee for his valuable comments on the first draft.

\section{AR principle in derived category}

Let  $R$  be a commutative Noetherian ring. 
We denote by  $\D = D( R)$  the full derived category of $R$. 
Note that the objects of  $\D$  are chain complexes over  $R$, which we denote by the cohomological notation such as 
$$
X = ( \cdots \to X ^{n-1} \to X^{n} \to X^{n+1} \to \cdots).  
$$
It should be noted that  $\D$  has a structure of triangulated category with shift functor, denoted by $X\mapsto X[1]$.

Recall that a full subcategory  $\L$ of  $\D$  is called a {\it localizing}  subcategory if it is a triangulated subcategory and it is closed under direct sums and direct summands. 
By Bousfield theorem \cite[Theorem 2.6]{N}, the natural inclusion  $i : \L \mono \D$ has a right adjoint functor $\gamma: \D \to \L$, i.e. 
$$
\Hom _{\D} ( iX, Y) \cong \Hom _{\L} (X, \gamma Y),
$$  
for all $X \in \L$  and  $Y \in \D$.

For a chain complex  $X$  we define the small support  $\supp (X)$ to be the set of prime ideals  $\p$  such that  $X \Lotimes _R \ \kappa (\p) \not= 0$, where  $\kappa (\p) = R_{\p}/\p R_{\p}$. 
(Cf. \cite{F}.) 
The small support of a full subcategory $\L$  of  $\D$  is the union of all the small supports of objects of  $\L$, so that  
$$
\supp (\L) = \{ \p \in \Spec (R) \ |\ X \Lotimes _R \ \kappa (\p) \not= 0 \ \text{for some} \ X \in \L\}. 
$$

For any subset  $W \subseteq \Spec (R)$, the full subcategory 
$$
\L_{W} = \{ X \in \D \ | \ \supp (X) \subseteq W\} 
$$
is a localizing subcategory of  $\D$.
The correspondences  $\L \mapsto \supp (\L)$ and  $W \mapsto \L_W$ yield a bijection between the set of localizing subcategories of $\D$ and the power set of $\Spec (R)$. 
This was proved by A.Neeman \cite[Theorem 2.8]{N}. 

We say that a localizing subcategory $\L _W$  is smashing if  $W$ is a specialization-closed subset of  $\Spec (R)$, and in this case, the functor  $\gamma : \D \to \L_W$  is nothing but the local cohomology functor   $\RGamma _W$. 
See \cite[Theorem 3.3]{N}.

\begin{remark}\label{rem1}
The big support  $\Supp (X)$  of a chain complex  $X \in \D$ is the set of prime ideals $\p$ of $R$  with the property  $X \Lotimes _R R_{\p} \not= 0$, or equivalently $H(X)_{\p} \not= 0$. 
In general it holds   
$$
\supp (X) \subseteq \Supp (X), 
$$
for all $X \in \D$. 
If $X$ belongs  to  $\D^{-}_{fg}( R)$, then we have  $\supp (X) = \Supp (X)$ which is a closed subset of  $\Spec (R)$.
Given a specialization-closed subset $W$ of $\Spec(R)$, a complex $X$ is an object of $\L_W$ if and only if the big support of X is contained in $W$ if and only if $\Supp H^i(X)\subseteq W$ for all $i\in\mathbb{Z}$.
See \cite{BIK}. 
\end{remark}

\begin{definition} 
Let  $X \in \D$ be a chain complex;  
$$
\cdots \longrightarrow X^n \overset{d^n}\longrightarrow X^{n+1} \overset{d^{n+1}}\longrightarrow X^{n+2} \longrightarrow \cdots. $$ 
For an integer  $n$  we define the truncations $\sigma _{> n}X$  and  $\sigma _{\leqq n}X$  as follows: 
$$
\begin{aligned}
&\sigma _{>n}X = \left( \cdots \to 0 \to \Im \ d^{n} \to X^{n+1} \overset{d^{n+1}}\longrightarrow X^{n+2} \to \cdots \right) \\ 
&\sigma _{\leqq n}X = \left(\cdots \to X^{n-2} \overset{d^{n-2}}\longrightarrow X^{n-1} \to \Ker \ d^{n} \to 0 \to \cdots \right) \\ 
\end{aligned}
$$
See \cite[Chapter 1; \S 7]{Hartshorne} for more detail. 
Note that there is an exact triangle in $\D$;
$$\sigma_{\leqq n}X \longrightarrow X \longrightarrow \sigma_{> n} X \longrightarrow \sigma_{\leqq n}X[1].$$
\end{definition}

Now the following theorem is a main theorem of this paper, which we call AR principle. 
Actually this is an equivalent version of the theorem in the Introduction.

\begin{theorem}\label{AR}
Let  $X, I$ be chain complexes in $\D$ and let $\L$  be a smashing subcategory of $\D$ with $\gamma : \D \to \L$  a right adjoint functor to the natural embedding $\iota : \L \mono \D$.
We assume the following conditions hold for some integer $n$; 
\begin{enumerate}
\item 
$I$  is a bounded injective complex, and right bounded at most in degree $n$. 
\item 
$\sigma_{\leqq -1 } X \in \L$.  
\end{enumerate}
Then the natural map $\gamma I \to I$  induces an isomorphism  
$$
\sigma _{> n} \RHom _R (X, \gamma I) \cong  \sigma _{> n} \RHom _R (X, I).  
$$
\end{theorem}

\begin{proof}
Since $\gamma$ is right adjoint to $\iota$, we have a counit morphism  $\iota\gamma I \to I$ in $\D$,  which induces the morphism 
$$
\RHom _R(X, \gamma I)  \to  \RHom _R (X, I). 
$$
To prove the theorem it is enough to shown that this morphism induces isomorphisms  
$$
H^i (\RHom _R(X, \gamma I))  \cong H^i( \RHom _R (X, I))
$$
for  $i >n$. 

Note that  $H^i( \RHom _R (X, I)) \cong  \Hom _{\D} (X, I[i])$, where  $[i]$  denotes the $i$ iterations of the shift functor $[1]$  in the triangulated category $\D$. 
Therefore, noting that  $I$  is a bounded injective complex, we see that an element $f$  of $H^i( \RHom _R (X, I))$  is a homotopy equivalence class of a chain map  $X \to I[i]$: 
$$
\begin{CD}
\cdots @>>> X^{-i-1} @>>> X^{-i} @>>> \cdots @>>> X^{-i+n} @>>> X^{-i+n+1} @>>> \cdots \\
@. @V{f^{-i-1}}VV @V{f^{-i}}VV @. @V{f^{-i+n}}VV @V{0}VV \\
\cdots @>>> I^{-1} @>>> I^{0} @>>> \cdots @>>> I^n @>>> 0
\end{CD}
$$
Since  $-i + n <0$, we have 
\begin{equation}\label{1}
\Hom _{\D} (X, I[i]) \cong \Hom _{\D} (\sigma _{\leqq -1} X, I[i]). 
\end{equation}
Now since  $\L$  is smashing, it forces that  $\gamma$  is of the form $\RGamma _W$  for a specialization-closed subset $W$ of $\Spec (R)$.  
Thus  $\gamma I$ is a subcomplex of $I$ and each term of  $\gamma I$ is also an injective module. 
As a consequence $\gamma I$ satisfies the same condition as  $I$. 
Therefore similar argument as above shows the isomorphism 
\begin{equation}\label{2}
\Hom _{\D} (X, \gamma I[i]) \cong \Hom _{\D} (\sigma _{\leqq -1} X, \gamma I[i]). 
\end{equation}
The right-hand sides in the equations (\ref{1}),(\ref{2}) are naturally isomorphic each other, since  $\sigma _{\leqq -1}X \in \L$. 
This completes the proof. 
\end{proof}

Theorem [AR Principle] stated in the introduction is a direct restatement of Theorem 2.3.
In fact, if $W$ is a specialization-closed subset of $\Spec(R)$ and if $\L=\L_W$, then it follows that $\gamma=\RGamma_W$, and the condition $(2)$ in Theorem 2.3 is equivalent to that $\Supp H^i(X)\subseteq W$ for $i<0$, by Remark 2.1.

\begin{remark}\label{LD}
We adopted such description of the AR principle as in Theorem \ref{AR}, because of its  similarity to the generalized version of local duality, that can be stated as follows: 

\vspace{6pt}{\it  
Let  $X, I$ be complexes in $\D$ and let $\L$  be a smashing subcategory of $\D$ with $\gamma : \D \to \L$  being as above. 
We assume the following conditions hold: 
\begin{enumerate}
\item 
$I$  is a bounded injective complex. 
\item 
$X \in \D^{-}_{fg}( R)$.  
\end{enumerate}
Then we have an isomorphism in  $\D$; 
$$
\RHom _R (X, \gamma I) \cong  \gamma \RHom _R (X, I).  
$$
}
\vspace{4pt}
\noindent
This version of local duality theorem was proposed by Hartshorne\cite[Chapter V, Theorem 6.2]{Hartshorne} and later generalized by Foxby\cite[Proposition 6.1]{F}. 
\end{remark}

\begin{question}
In Theorem \ref{AR} and Remark \ref{LD}, do the conclusions hold true if  $\L$  is not necessarily smashing but just localizing? 
\end{question}

\section{The case of isolated singularity}

Now in this section we assume that  $(R, \m)$  is a local ring of dimension $d$. 
We apply the AR principle to the following setting: 
\begin{itemize}
\item[$-$]
$W_0 = \{\m\}$, 
\item[$-$] 
$\L = \L_{W_0} = \{ X \in \D \ | \ \supp X \subseteq \{ \m \}\}$, and  
\item[$-$]
$I$  is a dualizing complex of  $R$.
\end{itemize}
We normalize $I$ so that it is of the form; 
$$
\begin{CD}
0 @>>> I ^0 @>>> I^1 @>>> \cdots @>>> I^d @>>> 0,  
\end{CD}
$$
where  $I^i = \bigoplus _{\dim R/\p = d-i}E_R(R/\p)$ for each $i$. 
(Cf. \cite{Hartshorne} or \cite{Sh}.)
In this case, since $\gamma = \RGamma _{\m}$, we have  $\gamma I = E [-d]$  where  $E = E_R(R/\m)$  is the injective hull of  $R/\m$.

For a chain complex  $X \in \D$ we denote 
$$
X^{\vee} = \RHom _R (X, E), \quad  X^{\dagger} = \RHom _R (X, I),
$$ 
which are respectively called the Matlis dual and the canonical (or Grothendieck) dual of  $X$. 
To all such situations, Theorem \ref{AR} can be applied directly and we get the following theorem. 

\begin{theorem}\label{AR1}
Let  $(R, \m)$  is a local ring of dimension $d$ as above. 
We assume  $X \in \D$  satisfies that  $\supp (\sigma _{\leqq -1}X) \subseteq \{\m\}$. Then we have an isomorphism 
$$
\sigma_{> 0} (X^{\vee})  \cong \sigma_{> d} (X^{\dagger}) [d].
$$
\end{theorem}
\indent
By Remark 2.1, the theorem can be stated in the following way: 

\vspace{6pt}
\indent
{\it If $\Supp H^i(X)\subseteq \{\m\}$ for all $i\leqq-1$, 
then $\Ext^j_R(X,E)\cong\Ext^{j+d}_R(X,I)$ for all $j>0$.}

\vspace{12pt}

Now assume that  $(R, \m)$ is a Cohen-Macaulay local ring which possesses canonical module  $\omega$. 
Note in this case that the dualizing complex $I$ is a minimal injective resolution of $\omega$. 
We denote the category of maximal Cohen-Macaulay modules over $R$  by $\CM (R)$. For a finitely generated $R$-module $M$ we write as  $\NF (M)$ the non-free locus of $M$, i.e. 
$$
\NF (M) = \{ \p \in \Spec (R) \ | \ M_{\p}\ \ \text{is not $R_{\p}$-free} \}.
$$
It is known and easily proved that $\NF (M)$ is a closed subset of  $\Spec (R)$ whenever  $M$  is finitely generated, since $\NF (M) = \Supp \, \Ext^1_R (M,\Omega M)$. 

We need to recall the definition of  the (Auslander) transpose for the corollary below. Let 
$
F_1 \overset{\partial}\to F_0 \to M \to 0
$
be a minimal free presentation of a finitely generated $R$-module $M$.
Then the transpose $\tr{M}$ is defined as $\mathrm{Coker}(\Hom(\partial,{R}))$.

Theorem \ref{AR1} implies the following result that generalizes a little the Auslander-Reiten duality mentioned in the beginning of the paper. 

\begin{corollary}\label{ARcor}
Let  $R$  be a Cohen-Macaulay local ring with canonical module and let $M, N \in \CM (R)$. 
Assume that  $\NF (M) \cap \NF (N) \subseteq \{ \m \}$. 
Then we have an isomorphism
$$
\pHom _R (M, N) ^{\vee} \cong \Ext^1_R (N, \tau M), 
$$
where 
$\tau M = [\Omega ^d(\tr M)]^{\dagger}$
\end{corollary}

\begin{proof}
Setting $X = \tr M \Lotimes _R N$, we see that the condition  $\NF (M) \cap \NF (N) \subseteq \{\m\}$  forces that  $\supp \ \sigma _{<0}X \subseteq \Supp \ \sigma _{<0}X \subseteq \{\m\}$.
Hence we can apply Theorem \ref{AR1} to $X$ and get an isomorphism
$$
H^{d+1} (X^{\dagger}) \cong H^{1}(X^{\vee}) \cong H^{-1}(X)^{\vee}.
$$
It is known that $H^{-1}(X) = \Tor _1^R (\tr M, N) \cong \pHom _R (M, N)$. 
See \cite[Lemma 3.9]{Ybook}.
On the other hand, 
since $X^{\dagger} = \RHom _R (\tr M \Lotimes _R N, I) \cong 
\RHom _R (N, [\tr M]^{\dagger}) \cong \RHom _R (\tr M, N^{\dagger})$,
we have 
$$
\begin{aligned}
H^{d+1}(X^{\dagger}) 
\cong \Ext ^{d+1}_R (\tr M, N^{\dagger}) 
\cong \Ext ^1_R (\Omega ^{d}\tr M, N^{\dagger})  & \cong \Ext ^1_R (N, [\Omega ^{d}\tr M]^{\dagger}). 
\end{aligned}
$$
\end{proof}

\begin{remark}
We remark form Corollary 3.2 that 
AR duality still holds even if $M$ is a finitely generated $R$-module but $N$ is not necessarily finitely generated. 
Suppose that $\NF(M) \cap \Supp(N) \subseteq \{\m\}$ and $H^i(N^\dagger)=0$ for $i>0$. In a similar way to Corollary 3.2, we can show that
$$\pHom (M,N)^{\vee} \cong \Ext^{1}_{R} (N,\tau M).$$
Note that if $R$ is a Cohen-Macaulay complete local ring and $N$ is a big Cohen-Macaulay module, it follows that $H^i(N^\dagger)=0$ for $i>0$. See \cite[Proposition 2.6]{PG}.
For example, $R=k[[x,y]]$ is a formal power series ring where $k$ is a field and $N=R\oplus E_R(R/(y))$. Assume that $M$ is a finitely generated $R$-module which is a locally free on the punctured spectrum. Since $N$ is a big Cohen-Macaulay module from \cite[Remark 3.3]{PG},
we obtain that
$$\pHom (M,E_R(R/(y)))^{\vee} \cong \Ext^1_{R} (E_R(R/(y)),\tau M).$$
\end{remark}
\section{The case of codimension one singular locus}

In this section  $(R, \m)$  always denotes a local ring of dimension $d$ as before. 
We consider the following conditions, in which we apply the AR principle \ref{AR}:  
\begin{itemize}
\item[$-$] 
$W _{1}= \{ \p \in \Spec(R)\ |\  \dim R/\p \leqq1\}$, 
\item[$-$]
$\L = \L _{W_1} = \{ X \in \D \ | \ \supp X \subseteq W _{1}\}$, and  
\item[$-$]
$R$  has a (normalized) dualizing complex $I$.
\end{itemize}

In this case,  since  $\gamma = \RGamma _{W_{1}}$,  it follows that $\gamma I$ is a two-term complex; 
\begin{equation}\label{twotermcpx}
\begin{CD}
0 @>>> I^{d-1} @>{\partial}>> I^{d} @>>> 0,  
\end{CD}
\end{equation}     
where 
$$
I ^{{d-1}}= \bigoplus  _{\dim R/\p =1} E_{R}(R/\p) =: J ,  \quad  I^{d} = E_{R}(R/\m) =: E. 
$$
We thus have a triangle in $\D$; 
$$
\begin{CD}
E[-d] @>>> \gamma I @>>> J[-d+1] @>{\partial [-d+1]}>> E[-d+1].  \\
\end{CD}
$$
Now let  $X \in \D$  and assume that  $\sigma _{\leqq -1}X \in \L$. 
It follows that there is a triangle in $\D$; 
$$
\begin{CD}
X^{\vee}[-d] @>>> \RHom _{R}(X, \gamma I) @>>>\RHom _{R}(X,  J)[-d+1] @>{\RHom (X,\partial)}>> X^{\vee}[-d+1].  \\
\end{CD}
$$
On the other hand, Theorem \ref{AR} says that there are isomorphisms 
$$
H ^{d+i} (X^{\dagger}) \cong H ^{d+i} (\RHom _{R}(X, \gamma I)), 
$$
for $i > 0$. 
Combining this isomorphism with the triangle above, we have the following proposition. 

\begin{proposition}\label{AR2}
Assume that  $X \in \D$ satisfies that  $\supp (\sigma_{\leqq -1}X) \subseteq W_{1}$. 
Then there is a long exact sequence of $R$-modules: 
$$
\begin{CD}
@. \Hom _{R}(H^{-1}(X), J)  @>{\Hom_R(H^{-1}(X), \partial)}>> H^{-1}(X)^{\vee}  @>>> H^{d+1}(X^{\dagger})  \\ 
@>>> \Hom _{R}(H^{-2}(X), J)  @>{\Hom_R(H^{-2}(X), \partial)}>> H^{-2}(X)^{\vee}  @>>> H^{d+2}(X^{\dagger})  \\
@>>> \cdots \\ 
@>>> \Hom _{R}(H^{-i}(X), J)  @>{\Hom_R(H^{-i}(X), \partial)}>> H^{-i}(X)^{\vee}  @>>> H^{d+i}(X^{\dagger})    \\
@>>> \cdots. \\
\end{CD}
$$   
\end{proposition}

This leads us to the following theorem that is more applicable to our computation. 
Recall that  $\D_{fg}( R)$ denotes the full subcategory of $\D$ consisting of all chain complexes whose cohomology modules are finitely generated $R$-modules.

\begin{theorem}\label{AR3}
Let $X$ be a chain complex in $\D_{fg}( R)$, and assume that  $\supp (\sigma_{\leqq -1}X) \subseteq W_{1}$. 
Then, for any $i>0$,  there is a short exact sequence 
$$
\begin{CD}  
 0@>>> H^{0}_{\m}(H^{-i}(X)) ^{\vee}@>>> H^{d+i}(X^{\dagger})^{\wedge} @>>> H^{1}_{\m} (H^{-i-1}(X))^{\vee} @>>> 0, 
\end{CD} \\ 
$$
and an isomorphism 
$$
H^{0}_{\m}(H^{-i}(X)) ^{\vee} \cong H_{\m}^0(H^{d+i}(X^{\dagger})). \\
$$
In the sequence above, ${}^{\wedge}$ denotes the $\m$-adic completion. 
\end{theorem}
Note from Remark \ref{rem1} that the assumption for $X$ in Theorem \ref{AR3} is precisely saying that $\Supp H^i(X) \subseteq W_{1}$ for $i<0$.

Before proving Theorem \ref{AR3} we note the following lemmas.
\begin{lemma}\label{lemma1}
Let $\partial : J \to E$ be the map in (\ref{twotermcpx}) above. 
Suppose we are given a finitely generated $R$-module $M$ such that  $\dim M \leqq 1$. 
Then the following hold. 
\begin{enumerate}
\item
There are isomorphisms of $\widehat{R}$-modules
$$
\Ker (\Hom _R(M, \partial)) ^{\wedge} \cong H_{\m}^1(M)^{\vee}, \quad  
\Cok (\Hom _R(M, \partial)) ^{\wedge} \cong H_{\m}^0(M)^{\vee}.
$$
\item
$H^1_{\m} (M) ^{\vee}$ is a Cohen-Macaulay $\widehat{R}$-module of dimension one, in particular, it holds that $H_{\m}^0(H^1_{\m} (M) ^{\vee})=0$.  
\end{enumerate}
\end{lemma}

\begin{proof}
(1)
Noting that $\dim M \leqq 1$, we have $\Hom _R (M, \bigoplus _{\dim R/ \p =i} E_R(R/\p)) = 0$ for all $i >1$. 
It hence follows the equalities 
$$
\Ker (\Hom _R(M, \partial))= H ^{d-1} (M^{\dagger}) \ \ \text{and} \quad  
\Cok (\Hom _R(M, \partial))= H ^{d}(M^{\dagger}).
$$
On the other hand the local duality theorem implies that 
$$ 
H ^{d-1} (M^{\dagger})^{\wedge} \cong H^1_{\m}(M) ^{\vee} \ \ \text{and} \quad   H ^{d} (M^{\dagger})^{\wedge} \cong H^0_{\m}(M) ^{\vee}. 
$$

(2)
We may assume $\dim M =1$. 
Note that  $\overline{M} = M/H_{\m}^0(M)$ is a one-dimensional Cohen-Macaulay $R$-module, and $H^1_{\m}(M)=H^1_{\m} (\overline{M})$. 
Replace $M$ by $\overline{M}$, and we may assume that $M$ is a one-dimensional Cohen-Macaulay module.
Then it is known that  $(M^{\dagger}) [d-1] = \Ext _R^{d-1}(M, I)$ is again one-dimensional Cohen-Macaulay, hence so is the completion $\Ext _R^{d-1}(M, I)^{\wedge}$. 
However it follows from the local duality that   $\Ext _R^{d-1}(M, I)^{\wedge} = H^1_{\m}(M)^{\vee}$. 
\end{proof}
\begin{lemma}\label{exercise}
Let  $M$  be a finitely generated module over a local ring  $(R, \m)$. 
Then the equality 
$$
H_{\m}^{0}(M) \cong H_{\m}^{0}(M^{\wedge})
$$
holds. 
\end{lemma}

\begin{proof}
Note that  $H^{0}_{\m}(M)$  is a unique submodule $N$ of $M$ such that 
$N$  is of finite length and  $M/N$  has no nontrivial submodule of finite length (or equivalently $\depth \ M/N >0$). 
Taking the $\m$-adic completion for modules in a short exact sequence 
$0 \to   H^{0}_{\m}(M) \to M \to \bar{M} \to 0$, and noting that the $\m$-adic topology on $H^{0}_{\m}(M)$ is discrete, we have an exact sequence 
$$
0 \to H^{0}_{\m}(M) \to M^{\wedge} \to \bar{M}^{\wedge} \to 0.
$$
Since  $\depth \bar{M}^{\wedge} >0$  as  $\depth \bar{M}>0$,  
we have the desired equality $H_{\m}^{0}(M) = H_{\m}^{0}(M^{\wedge})$. 
\end{proof}
Now we proceed to the proof of Theorem \ref{AR3}. 
It follows from Theorem \ref{AR2} that there is an exact sequence: 
$$
0 \to \Cok (\Hom _R(H^{-i}(X), \partial)) \to H^{d+i} (X^{\dagger}) \to \Ker (\Hom _R(H^{-i-1}(X), \partial)) \to 0, 
$$
for $i>0$. 
Since  $H^{-i}(X)$ and $H^{-i-1}(X)$ are finitely generated and their dimensions are at most one, we can apply Lemma \ref{lemma1} and get a short exact sequence 
$$
\begin{CD}  
 0@>>> H^{0}_{\m}(H^{-i}(X)) ^{\vee}@>>> H^{d+i}(X^{\dagger})^{\wedge} @>>> H^{1}_{\m} (H^{-i-1}(X))^{\vee} @>>> 0, 
\end{CD} \\ 
$$
as in Theorem \ref{AR2}. 
To show the isomorphism in Theorem \ref{AR3}, apply the functor $H_{\m}^{0}$  to this short exact and it is enough to notice from Lemma \ref{lemma1}(2) and \ref{exercise} that 
$
H_{\m}^0(H^{1}_{\m} (H^{-i-1}(X))^{\vee})= 0
$
and 
$
H^0_{\m}(H^{d+i} (X^{\dagger}))
\cong H^0_{\m}(H^{d+i} (X^{\dagger})^{\wedge}).
$
\qed 

Now let us assume that  $(R, \m)$ is Cohen-Macaulay and let  $M, N \in \CM ( R )$. 
We apply Theorem  \ref{AR3} above to  $X = \tr M \Lotimes _{R} N$, and we get a theorem of Iyama and Wemyss \cite{IW}. 
 
\begin{corollary}\label{AR4}
Let  $(R, \m)$ be a Cohen-Macaulay local ring and let  $M, N \in \CM ( R )$. 
We assume that  $\NF (M) \cap \NF (N) \subseteq W_{1}$. 
Then, for each $i>0$, there is a short exact sequence; 
$$
\begin{CD}  
 0 \to H^{0}_{\m}(\pHom _{R}(M, \Omega^{i-1}N)) ^{\vee}@>>> \Ext _{R}^{i}(N, \tau M)^{\wedge}  @>>> H^{1}_{\m} (\pHom _{R}(M, \Omega ^{i}N))^{\vee} \to 0,
\end{CD}
$$
and an isomorphism; 
$$
H^{0}_{\m}(\pHom _{R}(M, \Omega^{i-1}N)) ^{\vee} \cong H_{\m}^0 (\Ext _{R}^{i}(N, \tau M)). 
$$
\end{corollary}

  
\section{A remark on the Auslander-Reiten conjecture}

Now in this section we restrict ourselves to consider the case where  $R$ is Gorenstein. 
In this case it is easy to see that the syzygy functor  $\Omega : \pCM (R) \to  \pCM (R)$  is an auto-equivalence.
Hence, in particular, one can define the cosyzygy functor  $\Omega ^{-1}$  on  $\pCM (R)$  as the inverse of  $\Omega$.
We note from  \cite{RB} and \cite[2.6]{H}  that  $\pCM (R)$  is a triangulated category with shift functor  $[1] = \Omega ^{-1}$.
Note that  $\pHom _{R}(M, N) \cong  \Ext_{R}^{1}(M, \Omega ^{1}N)$ for all $M, N \in \CM ( R )$. 
Note also that, since  $R$  is Gorenstein, we have  
$$
\tau M = [\Omega ^{d}\tr M]^{\dagger} \cong \Omega ^{d-2}(M^{*})^{*} \cong M[d-2]. 
$$ 
Therefore Corollary \ref{ARcor} implies the fundamental duality. 
\begin{corollary}\label{ARclasic}
Let  $R$  be a Gorenstein local ring of dimension $d$.
Assume that  ${M}, {N} \in \CM (R)$ satisfy  $\NF (M) \cap \NF (N) \subseteq \{\m\}$. 
Then there is a functorial isomorphism
$$
\pHom _R (M, N)^{\vee} \cong \pHom _R (N, M[d-1]).
$$

\end{corollary}
Note that this is the case for any $M$ and $N$ if $R$ has at most an isolated singularity. 
On the other hand Theorem \ref{AR3} implies the following:  

\begin{corollary}\label{ARIW}
Let  $R$  be a Gorenstein local ring of dimension $d$.
Assume that  ${M}, {N} \in \CM (R)$ satisfy  $\NF (M) \cap \NF (N) \subseteq W_1$. 
Then there is a short exact sequence 
$$
 0 \to H^{0}_{\m}(\pHom _{R}(M, N[d-1])) ^{\vee}\to  \pHom _{R}(N, M)^{\wedge}  \to H^{1}_{\m} (\pHom _{R}(M, N[d-2]))^{\vee} \to 0.
$$
\end{corollary}

Araya \cite{A} shows that Corollary \ref{ARclasic} implies the Auslander-Reiten conjecture for Gorenstein rings with isolated singularity of dimension not less than $2$. 
Since the ring $R$ is Gorenstein, noting that a module $M$ is a maximal Cohen-Macaulay module if and only if $\Ext_{R}^{i}(M, R) =0$  for all $i>0$, 
we can state the AR conjecture for Gorenstein rings as follows: 
 
\begin{conjecture}\label{ARconjecture}
Let  $R$  be a Gorenstein local ring as above and  let  $M$ be in $\CM ( R )$. 
If  $\Ext_{R}^{i}(M, M) =0$  for all $i>0$, then  $M$ is a free $R$-module.  
\end{conjecture}
In fact, the assumption of the conjecture is equivalent to that 
$\pHom _{R}(M, M[i])=0$  for  $i>0$. 
 On the other hand $M$ is free if and only if  $\pHom _{R}(M, M) =0$. 
Therefore it is restated in the following form: 

\vspace{4pt} 
$-$ If $\pHom _{R}(M, M[i])=0$  for  $i>0$, then  $\pHom _{R}(M, M) =0$. 
\vspace{4pt} 

By virtue of Corollary \ref{ARclasic},  the conjecture is trivially true if $R$  is an isolated singularity and $d \geqq 2$. 
This is what Araya proved in his paper \cite{A}.

In contrast to this, we can prove the following theorem by using Corollary \ref{ARIW}. 
\begin{theorem} \label{last}
Let $(R, \m)$  be a Gorenstein local ring of dimension $d$ and let  $M, N \in \CM ( R )$. 
Assume the following conditions: 
\begin{enumerate}
\item 
$\NF (M) \cap \NF (N) \subseteq W_{1}$, 
\item 
$\depth \ \pHom _{R} (M, N [d-1]) > 0$, 
\item 
$\depth \ \pHom _{R}(M, N [d-2]) >1$.
\end{enumerate}
Then we have $\pHom (N, M) = 0$. 
\end{theorem} 

Note in the theorem that we adopt the convention that the depth of the zero module is $+\infty$, so that the conditions (2)(3) contain the case when 
$\pHom _{R} (M, N [d-1]) = \pHom _{R}(M, N [d-2]) = 0$.

The proof of Theorem \ref{last} is straightforward from Corollary \ref{ARIW}. 
In fact the assumptions for $M, N$ in the theorem imply the vanishing of the both ends in the short exact sequence in Corollary \ref{ARIW}, hence we have   $\pHom (N, M) = 0$.

The following is a direct consequence of the theorem, which generalizes the aforementioned Araya's result. 

\begin{corollary}\label{end}
Let $(R, \m)$  be a Gorenstein local ring of dimension $d$ and let  $M \in \CM ( R )$. 
Assume that $\NF (M) \subseteq W_{1}$. 
Furthermore assume 
$$
\depth \ \pHom _{R} (M, M [d-1]) > 0, \quad \depth \ \pHom _{R}(M, M [d-2]) >1.
$$
Then $M$ is a free $R$-module. 
\end{corollary} 

This result assures us that the AR conjecture \ref{ARconjecture} holds true if $\NF (M) \subseteq W_1$ and $d \geqq 3$. 
This is automatically the case, for example,  whenever  $R$ is a normal Gorenstein local domain of dimension $3$.

\vspace{36pt}

\end{document}